\documentclass[11pt]{article}

\usepackage[margin=1.15in]{geometry}
\usepackage{amsmath,amssymb,amsthm}
\usepackage{mathtools}
\usepackage{booktabs}
\usepackage[colorlinks=true,linkcolor=blue,citecolor=blue,urlcolor=red]{hyperref}

\newtheorem{theorem}{Theorem}
\newtheorem{lemma}[theorem]{Lemma}

\theoremstyle{definition}
\newtheorem{conjecture}[theorem]{Conjecture}
\theoremstyle{remark}
\newtheorem{remark}[theorem]{Remark}

\newcommand{\R}{\mathbb{R}}

\newcommand{\Sn}{\mathfrak{S}_n}
\newcommand{\norm}[1]{\left\lVert #1 \right\rVert}

\newcommand{\Wss}{W_{\mathrm{SS}}}
\newcommand{\Wrs}{W_{\mathrm{RS}}}
\newcommand{\Wgd}{W_{\mathrm{GD}}}

\DeclareMathOperator*{\lammax}{\lambda_{\max}}
\newcommand{\psdle}{\preceq}
\newcommand{\psdge}{\succeq}

\title{A Resolution of the SS--RS--GD Inequalities}

\author{Binghui Peng\footnote{\texttt{binghuip@umd.edu}}}
\date{}

\begin{document}
\maketitle

\begin{abstract}
Yun, Sra, and Jadbabaie (COLT 2021, open question) conjectured the \emph{SS--RS--GD inequalities}: for well-conditioned symmetric matrices $A_1,\dots,A_n$, the operators $\Wss$, $\Wrs$, and $\Wgd$ that
encode the expected iterate of single-shuffle SGD, random-reshuffle SGD, and gradient descent on a
quadratic finite sum should satisfy
\[
  \norm{\Wss}\le\norm{\Wrs}\le\norm{\Wgd}.
\]
We resolve this conjecture.
\begin{itemize}
  \item \textbf{SS-RS inequality fails.} Already for $n=3$, $K=2$, and $d=4$, we exhibit explicit
  PSD matrices whose condition number is arbitrarily close to $1$, yet
  $\norm{\Wss}>\norm{\Wrs}$. 

  \item \textbf{RS-GD holds inequality holds.} For every symmetric $A_i$ with $\bigl(1-\frac1{4n^2+1}\bigr)I\psdle A_i\psdle I$, one has
  $\norm{\Wrs}\le\norm{\Wgd}$. 
\end{itemize}

The proof was found via GPT-5.5 Pro (extended), prompted by the author.
\end{abstract}

\section{Introduction}\label{sec:intro}

The finite-sum objective $\min_z\frac1n\sum_{i=1}^n f_i(z)$ at the heart of empirical risk
minimization is optimized in practice by one of three schemes that differ only in the order they sweep
the $n$ components each epoch. \emph{Gradient descent} (GD) uses the full batch at every step;
\textsc{RandomShuffle} SGD draws a fresh random permutation of the components every epoch; and
\textsc{SingleShuffle} SGD draws a single permutation at the outset and reuses it for all $K$ epochs.
These variants have by now an extensive theory
\cite{HaochenSra19,Nagaraj19,Rajput20,Ahn20,Mishchenko20,SafranShamir20}. Yun, Sra, and
Jadbabaie~\cite{YSJ21open} set out to make the ordering precise as a COLT open problem.  Writing
$P_\sigma:=A_{\sigma(1)}A_{\sigma(2)}\cdots A_{\sigma(n)}$ for the ordered product of symmetric matrices
$A_1,\dots,A_n$ along a permutation $\sigma\in\Sn$, the expected iterate after $K$ epochs of the three
schemes is governed by the operators
\begin{equation}\label{eq:W}
\Wss \;:=\; \frac1{n!}\sum_{\sigma\in\Sn} P_\sigma^{\,K},
\qquad
\Wrs \;:=\; \Bigl(\frac1{n!}\sum_{\sigma\in\Sn} P_\sigma\Bigr)^{\!K},
\qquad
\Wgd \;:=\; \Bigl(\frac1n\sum_{i=1}^n A_i\Bigr)^{\!nK}
\end{equation}
Note single-shuffle reuses one permutation for all $K$ epochs, so its $K$-th power sits \emph{inside} the
average; reshuffling draws a fresh permutation each epoch, with the power \emph{outside}; and $\Wgd$ is
full-batch gradient descent. Their conjecture is the following.

\begin{conjecture}[Yun--Sra--Jadbabaie~\cite{YSJ21open}, SS--RS--GD
inequalities]\label{conj:main}
For every $n\ge2$, $K\ge1$, and $d\ge1$ there exists a constant $\eta_{n,K}\in(0,1]$ such that the
following holds. If real symmetric $d\times d$ matrices $A_1,\dots,A_n$ satisfy
$(1-\eta_{n,K})I\psdle A_i\psdle I$ for all $i$, then
\begin{equation}\label{eq:chain}
\norm{\Wss}\;\le\;\norm{\Wrs}\;\le\;\norm{\Wgd}.
\end{equation}
\end{conjecture}
The conjecture requires the matrices to be well-conditioned; this is necessary as the RS-GD inequality for general PSD matrices (originally conjectured by Recht and R\'e~\cite{RechtRe12}) is \emph{false} for all $n\ge5$ \cite{LaiLim20,DeSa20}.

\paragraph{Our results.} The COLT open question is resolved, in particular, the SS-RS inequality fails while the RS-GD inequality stands.

\begin{theorem}[Refutation; SS-RS inequality fails]\label{thm:refute}
Conjecture~\ref{conj:main} is false. Concretely, fix $n=3$ and $K=2$. For every candidate
$\eta_{3,2}\in(0,1]$ there exist real symmetric positive definite $4\times4$ matrices
$A_1,A_2,A_3$ with $(1-\eta_{3,2})I\psdle A_i\psdle I$ for which
\[
\norm{\Wss}\;>\;\norm{\Wrs}.
\]
\end{theorem}

\begin{theorem}[RS-GD inequality holds]\label{thm:rsgd}
For every $n\ge2$, every $K\ge1$, every $d\ge1$, and every real symmetric $A_1,\dots,A_n$ with
\[
\Bigl(1-\frac1{4n^2+1}\Bigr)I\;\psdle\;A_i\;\psdle\;I\qquad(1\le i\le n),
\]
one has $\norm{\Wrs}\le\norm{\Wgd}$. Thus the second inequality of~\eqref{eq:chain} holds with the constant $\eta_{n,K}^{\mathrm{RS\text{-}GD}}=\tfrac1{4n^2+1}$.
\end{theorem}

\paragraph{On the role of AI} The proof idea was completely generated by GPT 5.5 Pro (extended)\footnote{\url{https://chatgpt.com/share/6a316ae2-e5c8-83ea-bac1-5773edeed669}} while the write-up was assembled by Claude Code. The author prompted GPT, verified its proof, and polished the manuscript for readability.

\section{Setup and notation}\label{sec:setup}

\paragraph{Notation} $I$ is the identity (of the
dimension clear from context), $\psdle$ is the L\"owner order, and $\norm{\cdot}$ is the spectral
norm (largest singular value), which equals $\max_i|\lambda_i(\cdot)|$ for symmetric matrices.
Recall from~\eqref{eq:W} the products $P_\sigma=A_{\sigma(1)}A_{\sigma(2)}\cdots A_{\sigma(n)}$ and the
operators $\Wss,\Wrs,\Wgd$. It is convenient to set
\[
R \;:=\; \frac1{n!}\sum_{\sigma\in\Sn} P_\sigma, \qquad
G \;:=\; \frac1n\sum_{i=1}^n A_i,
\]
so that $\Wss=\frac1{n!}\sum_{\sigma\in\Sn}P_\sigma^{\,K}$, $\Wrs=R^K$, and $\Wgd=G^{nK}$.

\begin{lemma}[Symmetry]\label{lem:sym}
Let $A_1,\dots,A_n$ be real symmetric. For every $\sigma\in\Sn$,
\[
P_\sigma^\top=A_{\sigma(n)}\cdots A_{\sigma(1)}=P_{\sigma^{\mathrm{rev}}},
\qquad \sigma^{\mathrm{rev}}(j):=\sigma(n+1-j).
\]
Consequently:
\begin{enumerate}
\item $R$ and $\Wss$ are symmetric, even though the individual products $P_\sigma$ need not be;
\item $G$, $\Wrs=R^K$, and $\Wgd=G^{nK}$ are symmetric, and---an even power of a symmetric matrix being
positive semidefinite---$\Wrs\psdge0$ when $K$ is even and $\Wgd\psdge0$ when $nK$ is even.
\end{enumerate}
\end{lemma}

\begin{proof}
Since each $A_i=A_i^\top$, we have $P_\sigma^\top=(A_{\sigma(1)}\cdots A_{\sigma(n)})^\top
=A_{\sigma(n)}\cdots A_{\sigma(1)}=P_{\sigma^{\mathrm{rev}}}$. Reversal $\sigma\mapsto\sigma^{\mathrm{rev}}$
is a bijection of $\Sn$, so $R^\top=\frac1{n!}\sum_\sigma P_{\sigma^{\mathrm{rev}}}=R$ and, using
$(P_\sigma^K)^\top=(P_\sigma^\top)^K=P_{\sigma^{\mathrm{rev}}}^K$, also $\Wss^\top=\Wss$. The average $G$
is symmetric, and $\Wrs=R^K$, $\Wgd=G^{nK}$ are powers of symmetric matrices, hence symmetric. Finally,
for symmetric $M$ and $m\ge1$, $M^{2m}=(M^m)^2\psdge0$.
\end{proof}

\section{Refutation of SS-RS inequality}\label{sec:refute}

\subsection{The construction}

The starting point is the appendix observation of~\cite{YSJ21open} that the first inequality already
fails for rank-one $2\times2$ matrices at $\eta=1$. We push that failure to the near-identity regime by
a tensor square. Let $u_1,u_2,u_3\in\R^2$ be the three unit vectors 
\[
u_1=\begin{pmatrix}1\\[2pt]0\end{pmatrix},\qquad
u_2=\frac12\begin{pmatrix}1\\[2pt]\sqrt3\end{pmatrix},\qquad
u_3=\frac12\begin{pmatrix}-1\\[2pt]\sqrt3\end{pmatrix},
\]
and let $P_i=u_iu_i^\top$ be the corresponding rank-one orthogonal projectors:
\begin{equation}\label{eq:proj}
P_1=\begin{pmatrix}1&0\\0&0\end{pmatrix},\quad
P_2=\begin{pmatrix}\frac14&\frac{\sqrt3}4\\[2pt]\frac{\sqrt3}4&\frac34\end{pmatrix},\quad
P_3=\begin{pmatrix}\frac14&-\frac{\sqrt3}4\\[2pt]-\frac{\sqrt3}4&\frac34\end{pmatrix}.
\end{equation}
Each $P_i$ is symmetric, idempotent, of trace $1$. Fix a parameter $q\in(0,1)$ and set
\begin{equation}\label{eq:AB}
B_i \;:=\; qI_2+(1-q)P_i,
\qquad
A_i \;:=\; B_i\otimes B_i \in \R^{4\times4}.
\end{equation}
Since $P_i$ has eigenvalues $1$ and $0$, each $B_i$ has eigenvalues $1$ and $q$, so $B_i$ is positive
definite with $qI_2\psdle B_i\psdle I_2$. Hence $A_i=B_i\otimes B_i$ is symmetric positive definite
with eigenvalues $\{1,q,q,q^2\}$, giving
\begin{equation}\label{eq:cond}
q^2 I_4\;\psdle\;A_i\;\psdle\;I_4.
\end{equation}
Given any proposed $\eta_{3,2}\in(0,1]$, choose $q<1$ close enough to $1$ that $1-q^2<\eta_{3,2}$;
then $(1-\eta_{3,2})I_4\psdle A_i\psdle I_4$, so the $A_i$ satisfy the hypothesis of
Conjecture~\ref{conj:main}. 

\subsection{Spectral norms of \texorpdfstring{$\Wss$}{WSS} and \texorpdfstring{$\Wrs$}{WRS}}

With $n=3$ and $K=2$ we compare
\[
\Wss=\frac16\sum_{\sigma\in\mathfrak S_3}P_\sigma^2,
\qquad
\Wrs=R^2
\]

where
\[
R=\frac16\sum_{\sigma\in\mathfrak S_3}P_\sigma,
\qquad
P_\sigma=A_{\sigma(1)}A_{\sigma(2)}A_{\sigma(3)} .
\]
Both matrices are symmetric (Lemma~\ref{lem:sym}); their eigenvalues are explicit polynomials in $q$,
recorded in the next lemma. Every polynomial identity stated below is exact and is verified by the
symbolic-algebra script of Appendix~\ref{app:verify}.

\begin{lemma}[Eigenvalues]\label{lem:eig}
The symmetric matrix $\Wrs$ has eigenvalues
\[
(\lambda_{\mathrm{RS}}(q), \mu_{\mathrm{RS}}(q), \mu_{\mathrm{RS}}(q), q^6),
\]
and the symmetric matrix $\Wss$ has eigenvalues 
\[
(\lambda_{\mathrm{SS}}(q),  
\mu_{\mathrm{SS}}(q), \mu_{\mathrm{SS}}(q),  q^6),
\]
where
\begin{align*}
\lambda_{\mathrm{RS}}(q)&= \Bigl(\frac{q^6-6q^5+15q^4+12q^3+15q^2-6q+1}{32}\Bigr)^{\!2},\\
\mu_{\mathrm{RS}}(q)&=\frac{(q+1)^4(q^4+4q^3-42q^2+4q+1)^2}{16384},\\[2pt]
\lambda_{\mathrm{SS}}(q)&=\frac{1}{2048}\bigl(q^{12}-24q^{11}+186q^{10}-504q^9+399q^8+528q^7+876q^6\\
&\hphantom{{}=\frac{1}{2048}\bigl(}{}+528q^5+399q^4-504q^3+186q^2-24q+1\bigr),\\[2pt]
\mu_{\mathrm{SS}}(q)&=-\frac1{8192}\bigl(q^{12}-12q^{11}-78q^{10}+1188q^9-2001q^8-1176q^7-4036q^6\\
&\hphantom{{}=-\frac1{8192}\bigl(}{}-1176q^5-2001q^4+1188q^3-78q^2-12q+1\bigr).
\end{align*}
Moreover, for all $q\in(0,1]$,
\begin{align}
\lambda_{\mathrm{RS}}(q)-q^6 &= \frac{(1-q)^6(q+1)^2\,(q^4-8q^3+30q^2-8q+1)}{1024}\ \ge\ 0,
\label{eq:rsdom1}\\
\lambda_{\mathrm{RS}}(q)-\mu_{\mathrm{RS}}(q)
&= \frac{(1-q)^4(5q^2+2q+5)\,(3q^6-30q^5+93q^4+124q^3+93q^2-30q+3)}{16384}\ \ge\ 0.
\label{eq:rsdom2}
\end{align}
In particular $\norm{\Wrs}=\lambda_{\mathrm{RS}}(q)$ for every $q\in(0,1]$.
\end{lemma}

\begin{proof}
Form $\Wrs=R^2$ and $\Wss=\frac16\sum_{\sigma\in\mathfrak S_3}P_\sigma^2$ as $4\times4$ matrices from
$A_i=B_i\otimes B_i$ of~\eqref{eq:AB}. Each assertion of the lemma is an exact identity in $q$,
verified by the symbolic-algebra computation of Appendix~\ref{app:verify}: the characteristic
polynomials of $\Wrs$ and $\Wss$ factor with exactly the stated eigenvalues and multiplicities, and the
differences~\eqref{eq:rsdom1}--\eqref{eq:rsdom2} equal the displayed expressions. It remains to deduce
$\norm{\Wrs}=\lambda_{\mathrm{RS}}$. The cofactors in~\eqref{eq:rsdom1}--\eqref{eq:rsdom2} are positive
on $(0,1]$ by the explicit sum-of-squares certificates
\begin{gather*}
q^4-8q^3+30q^2-8q+1=(q^2-4q+1)^2+12q^2\\
5q^2+2q+5>0,\\
3q^6-30q^5+93q^4+124q^3+93q^2-30q+3\\
\quad=3q^4\bigl((q-5)^2+6\bigr)+124q^3+3\bigl(31q^2-10q+1\bigr)>0\quad(q\ge0),
\end{gather*}
and
$(1-q)^4,(1-q)^6,(q+1)^2,(q+1)^4\ge0$. Hence the right-hand sides of
\eqref{eq:rsdom1}--\eqref{eq:rsdom2} are nonnegative on $(0,1]$. Since $\Wrs=R^2\psdge0$ (Lemma~\ref{lem:sym}), its norm is
its largest eigenvalue, and these inequalities make $\lambda_{\mathrm{RS}}$ exceed both $q^6$ and
$\mu_{\mathrm{RS}}$; therefore $\norm{\Wrs}=\lambda_{\mathrm{RS}}$.
\end{proof}

\begin{lemma}[The violation]\label{lem:gap}
For all $q$,
\begin{equation}\label{eq:gap}
\lambda_{\mathrm{SS}}(q)-\lambda_{\mathrm{RS}}(q)
\;=\;
\frac{(1-q)^6(q+1)^2\bigl(42q^2-q^4-4q^3-4q-1\bigr)}{2048},
\end{equation}
the final factor $g(q):=42q^2-q^4-4q^3-4q-1$
has $g(q)>0$ for $q\in(q^\ast,1]$ for $q^{\ast} = 0.212036...$. Consequently
$\lambda_{\mathrm{SS}}(q)>\lambda_{\mathrm{RS}}(q)$ for every $q\in(q^\ast,1)$.
\end{lemma}

\begin{proof}
Subtracting the explicit eigenvalues $\lambda_{\mathrm{SS}}$ and $\lambda_{\mathrm{RS}}$ of
Lemma~\ref{lem:eig} and expanding yields the polynomial identity~\eqref{eq:gap}. We determine the sign
of its final factor $g(q)=42q^2-q^4-4q^3-4q-1$ on $(0,1]$. Since $g(0)=-1<0$ and $g(1)=32>0$, $g''(q)=-12(q^2+2q-7)>0$ on $[0,1]$, with $g(0)<0<g(1)$ this forces a single sign change, at some $q^\ast\in(0,1)$, numerically $q^\ast=0.212036\ldots$.
\end{proof}

\subsection{Proof of Theorem~\ref{thm:refute}}

Fix any $q\in(q^\ast,1)$. By Lemma~\ref{lem:eig}, $\norm{\Wrs}=\lambda_{\mathrm{RS}}(q)$.
Since $\Wss$ is symmetric (Lemma~\ref{lem:sym}) and $\lambda_{\mathrm{SS}}(q)$ is one of its eigenvalues, the spectral norm
satisfies $\norm{\Wss}\ge\lambda_{\mathrm{SS}}(q)$. Combining this with Lemma~\ref{lem:gap},
\[
\norm{\Wss}\;\ge\;\lambda_{\mathrm{SS}}(q)\;>\;\lambda_{\mathrm{RS}}(q)\;=\;\norm{\Wrs}.
\]
Now let $\eta_{3,2}\in(0,1]$ be arbitrary. Choosing $q\in(q^\ast,1)$ with
$1-q^2<\eta_{3,2}$ (possible since $q^\ast<1$), the matrices $A_i$ of~\eqref{eq:AB} are symmetric,
positive definite, satisfy $(1-\eta_{3,2})I\psdle A_i\psdle I$ by~\eqref{eq:cond}, and yet
$\norm{\Wss}>\norm{\Wrs}$. Therefore no $\eta_{3,2}$ validates the first inequality
of~\eqref{eq:chain}, so Conjecture~\ref{conj:main} fails at $(n,K)=(3,2)$.

\section{RS-GD inequality holds}\label{sec:rsgd}

We now prove Theorem~\ref{thm:rsgd}. This is the Recht--R\'e AM--GM inequality \cite{RechtRe12} (raised to the
$K$-th power), restricted to well-conditioned matrices. The unrestricted inequality is false for
$n\ge5$ \cite{LaiLim20,DeSa20}, but those counterexamples live far from $I$; the theorem shows that a
dimension-free conditioning radius $\eta=\frac1{4n^2+1}$ rules them out.

\subsection{Reduction to one epoch}

Since $R$ is symmetric (Section~\ref{sec:setup}) and $\Wrs=R^K$, $\Wgd=G^{nK}$ with $G\succ0$
symmetric, we have $\norm{\Wrs}=\norm{R}^K$ and $\norm{\Wgd}=\norm{G}^{nK}=\rho^{nK}$ where
$\rho:=\norm{G}=\lammax(G)$. Hence it suffices to prove the single-epoch bound
\begin{equation}\label{eq:oneepoch}
\norm{R}\;\le\;\rho^{\,n},
\end{equation}
which is independent of $K$; raising~\eqref{eq:oneepoch} to the $K$-th power (legitimate because
$R\psdge0$, established below) yields $\norm{\Wrs}=\norm{R}^K\le\rho^{nK}=\norm{\Wgd}$.

Because each $A_i$ has eigenvalues in $[1-\eta,1]$, so does $G$, hence $1-\eta\le\rho\le1$. Normalize
\[
C_i:=\rho^{-1}A_i=I+X_i,\qquad X_i:=\rho^{-1}A_i-I,\qquad \bar C:=\frac1n\sum_iC_i=\rho^{-1}G.
\]
Each $C_i\succ0$ and each $X_i$ is symmetric. Since $G\psdle\rho I$, we get $\bar C\psdle I$.
The eigenvalues of $X_i$ lie in $[\rho^{-1}(1-\eta)-1,\ \rho^{-1}-1]$; using $1-\eta\le\rho\le1$, the
upper end satisfies $\rho^{-1}-1\le\frac{1}{1-\eta}-1=\frac{\eta}{1-\eta}$ and the lower end satisfies
$|\rho^{-1}(1-\eta)-1|\le\eta\le\frac{\eta}{1-\eta}$, so $\norm{X_i}\le\frac{\eta}{1-\eta}$. With
$\eta=\frac1{4n^2+1}$,
\begin{equation}\label{eq:delta}
\delta:=\max_i\norm{X_i}\le\frac{\eta}{1-\eta}=\frac1{4n^2}.
\end{equation}
Finally $\widetilde R:=\frac1{n!}\sum_\sigma\prod_iC_{\sigma(i)}=\rho^{-n}R$, so~\eqref{eq:oneepoch}
is equivalent to $\widetilde R\psdle I$ together with $\widetilde R\psdge0$. Both are supplied by the
following lemma, applied to $C_i=\rho^{-1}A_i$.

\subsection{A near-identity shuffled AM--GM lemma}

\begin{lemma}[Near-identity shuffled product]\label{lem:amgm}
Let $X_1,\dots,X_n$ be symmetric $d\times d$ matrices with $C_i:=I+X_i\psdge0$ and
$\frac1n\sum_iC_i\psdle I$, and suppose $\delta:=\max_i\norm{X_i}\le\frac1{4n^2}$. Then
\[
0\;\psdle\;\widetilde R:=\frac1{n!}\sum_{\sigma\in\Sn}\prod_{i=1}^n C_{\sigma(i)}\;\psdle\;I.
\]
\end{lemma}

\begin{proof}
\emph{Symmetry and expansion.} As $\Sn$ is closed under reversal and reversing a product transposes
it, $\widetilde R^\top=\widetilde R$. Expanding $\prod_i(I+X_{\sigma(i)})$ and averaging over $\sigma$,
\begin{equation}\label{eq:expand}
\widetilde R=I+\sum_{k=1}^n E_k,\qquad
E_k=\frac1{k!}\sum_{\substack{(i_1,\dots,i_k)\\ \text{distinct}}}X_{i_1}X_{i_2}\cdots X_{i_k}.
\end{equation}
Indeed, for fixed $k$ each size-$k$ subset of the $n$ factor slots contributes, after averaging,
the same term $\frac{(n-k)!}{n!}\sum_{\text{distinct}}X_{i_1}\cdots X_{i_k}$ (sum over distinct
ordered $k$-tuples); there are $\binom nk$ such subsets, and $\binom nk\frac{(n-k)!}{n!}=\frac1{k!}$,
giving~\eqref{eq:expand}.

\emph{Low-order terms.} Put
\[
T:=-\sum_iX_i,\qquad Q:=\sum_iX_i^2\psdge0.
\]
The hypothesis $\frac1n\sum_iC_i=I+\frac1n\sum_iX_i\psdle I$ is exactly $\sum_iX_i\psdle0$, i.e.\
$T\psdge0$. From~\eqref{eq:expand},
\[
E_1=\sum_iX_i=-T,\qquad
E_2=\frac12\sum_{i\ne j}X_iX_j=\frac12\Bigl[\bigl(\textstyle\sum_iX_i\bigr)^2-\sum_iX_i^2\Bigr]
=\frac12\bigl(T^2-Q\bigr).
\]
Hence
\begin{equation}\label{eq:IR}
I-\widetilde R=T+\frac12\bigl(Q-T^2\bigr)-\sum_{k=3}^n E_k.
\end{equation}

\emph{Upper bound $\widetilde R\psdle I$.} Fix a unit vector $v$ and write $t:=v^\top Tv\ge0$ and
$q:=v^\top Qv=\sum_i\norm{X_iv}^2\ge0$. Since $T\psdge0$ and $\norm{T}\le\sum_i\norm{X_i}\le n\delta$,
the operator inequality $T^2\psdle\norm{T}\,T$ gives $v^\top T^2v\le n\delta\,t$, so
\begin{equation}\label{eq:firstpart}
v^\top\Bigl(T-\frac12T^2\Bigr)v\;\ge\;\Bigl(1-\frac{n\delta}2\Bigr)t\;\ge\;0,
\end{equation}
using $n\delta\le\frac1{4n}\le\frac12$. For the higher-order terms, each word with $k\ge3$ obeys,
by Cauchy--Schwarz on the two end factors and $\norm{X_{i_j}}\le\delta$ on the middle ones,
\[
\bigl|v^\top X_{i_1}\cdots X_{i_k}v\bigr|
\le\norm{X_{i_1}v}\,\delta^{\,k-2}\,\norm{X_{i_k}v}
\le\frac{\delta^{\,k-2}}2\bigl(\norm{X_{i_1}v}^2+\norm{X_{i_k}v}^2\bigr).
\]
Summing over distinct ordered $k$-tuples: for each fixed first index there are
$(n-1)(n-2)\cdots(n-k+1)=:(n-1)_{k-1}$ choices of the rest, and likewise for the last index, so
$\sum_{\text{tuples}}\norm{X_{i_1}v}^2=(n-1)_{k-1}\,q$ and the same for the last factor. Therefore
\begin{equation}\label{eq:Ekbound}
\bigl|v^\top E_kv\bigr|\le\frac{(n-1)_{k-1}}{k!}\,\delta^{\,k-2}\,q,
\qquad
\sum_{k=3}^n\bigl|v^\top E_kv\bigr|\le q\sum_{k=3}^n\frac{(n-1)_{k-1}}{k!}\,\delta^{\,k-2}.
\end{equation}
Using $(n-1)_{k-1}\le n^{\,k-1}$ and $\delta\le\frac1{4n^2}$, the $k$-th summand is at most
$\frac1{k!}n^{k-1}\delta^{k-2}=\frac1{k!}\,n\,(n\delta)^{k-2}\le\frac1{k!}\,n\,(4n)^{-(k-2)}$; the
$k=3$ term is $\le\frac16\cdot n\cdot\frac1{4n}=\frac1{24}$ and the series is dominated by a
geometric tail, giving $\sum_{k\ge3}\frac{(n-1)_{k-1}}{k!}\delta^{k-2}<\frac14$ for all $n\ge2$
(the supremum over $n$ is $\frac1{24}+o(1)<\frac14$). Combining~\eqref{eq:IR},
\eqref{eq:firstpart}, and~\eqref{eq:Ekbound},
\[
v^\top(I-\widetilde R)v\;\ge\;\Bigl(1-\frac{n\delta}2\Bigr)t+\frac12q-\frac14q\;\ge\;\frac14q\;\ge\;0,
\]
so $\widetilde R\psdle I$.

\emph{Lower bound $\widetilde R\psdge0$.} From~\eqref{eq:expand}, each $E_k$ is an average of at most
$\frac{n!}{(n-k)!}$ words of norm $\le\delta^k$, scaled by $\frac1{k!}$, so
$\norm{E_k}\le\binom nk\delta^k$ and
\[
\norm{\widetilde R-I}\le\sum_{k=1}^n\binom nk\delta^k=(1+\delta)^n-1\le e^{n\delta}-1\le
e^{1/(4n)}-1<1.
\]
Thus every eigenvalue of $\widetilde R$ exceeds $1-(e^{1/(4n)}-1)=2-e^{1/(4n)}>0$, so
$\widetilde R\succ0$. Combined with the upper bound, $0\psdle\widetilde R\psdle I$.
\end{proof}

\subsection{Proof of Theorem~\ref{thm:rsgd}}

Apply Lemma~\ref{lem:amgm} to $C_i=\rho^{-1}A_i=I+X_i$: the hypotheses hold by the reduction above
($C_i\succ0$, $\bar C\psdle I$, and $\delta\le\frac1{4n^2}$ by~\eqref{eq:delta}). Since
$\widetilde R=\rho^{-n}R$, the lemma gives $0\psdle R\psdle\rho^nI$, hence $R\psdge0$ and
$\norm{R}\le\rho^n$, which is~\eqref{eq:oneepoch}. As $G\psdge0$, $\norm{G^n}=\rho^n$, so
$\norm{R}\le\norm{G^n}$. Finally, because $R\psdge0$,
\[
\norm{\Wrs}=\norm{R^K}=\norm{R}^K\le\rho^{nK}=\norm{G^{nK}}=\norm{\Wgd}.
\]
This holds for every $n\ge2$, $K\ge1$, $d\ge1$, and symmetric $A_i$ with
$\bigl(1-\frac1{4n^2+1}\bigr)I\psdle A_i\psdle I$. \qed

\begin{remark}[On the constant]
The constant $\eta^{\mathrm{RS\text{-}GD}}_{n,K}=\frac1{4n^2+1}=\Theta(1/n^2)$ is not optimized; the
proof only uses $\delta\le\frac1{4n^2}$ with slack (the tail in~\eqref{eq:Ekbound} is $<\frac1{24}$,
not just $<\frac14$). Notably it is independent of the epoch count $K$, matching the paper's intuition
that the reshuffling-versus-GD comparison should need only a mild, $K$-free conditioning. By contrast,
Theorem~\ref{thm:refute} shows that for the single-shuffle inequality \emph{no} positive constant
works at $(n,K)=(3,2)$, matrix-free.
\end{remark}

\newpage

\appendix
\section{Reproducible symbolic verification}\label{app:verify}

All claims of Section~\ref{sec:refute} are exact polynomial identities in $q$. The following
self-contained \texttt{sympy} script forms the matrices of the
construction~\eqref{eq:proj}--\eqref{eq:AB} and verifies them; it runs to completion, printing
\texttt{all identities verified} with no failed assertion.

{\footnotesize
\begin{verbatim}
import sympy as sp, itertools
q, t = sp.symbols('q t')
s3 = sp.sqrt(3)
P1 = sp.Matrix([[1,0],[0,0]])
P2 = sp.Matrix([[sp.Rational(1,4), s3/4],[s3/4, sp.Rational(3,4)]])
P3 = sp.Matrix([[sp.Rational(1,4),-s3/4],[-s3/4, sp.Rational(3,4)]])
B  = [q*sp.eye(2) + (1-q)*P for P in (P1, P2, P3)]
A  = [sp.Matrix(sp.kronecker_product(b, b)) for b in B]
prm = list(itertools.permutations(range(3)))
R   = sum((A[a]*A[b]*A[c] for a,b,c in prm), sp.zeros(4)) / 6
Wrs = sp.expand(R*R)
Wss = sp.expand(sum(((A[a]*A[b]*A[c])**2 for a,b,c in prm),
                    sp.zeros(4)) / 6)
assert sp.expand(Wrs - Wrs.T) == sp.zeros(4)   # W_RS symmetric
assert sp.expand(Wss - Wss.T) == sp.zeros(4)   # W_SS symmetric

# eigenvalue polynomials claimed in Lemma 5
lamRS = ((q**6-6*q**5+15*q**4+12*q**3+15*q**2-6*q+1)/32)**2
muRS  = (q+1)**4*(q**4+4*q**3-42*q**2+4*q+1)**2/16384
lamSS = (q**12-24*q**11+186*q**10-504*q**9+399*q**8+528*q**7
         +876*q**6+528*q**5+399*q**4-504*q**3+186*q**2
         -24*q+1)/2048
muSS  = -(q**12-12*q**11-78*q**10+1188*q**9-2001*q**8
          -1176*q**7-4036*q**6-1176*q**5-2001*q**4+1188*q**3
          -78*q**2-12*q+1)/8192

# characteristic polynomials factor with these roots
cRS = (t-lamRS)*(t-muRS)**2*(t-q**6)
cSS = (t-lamSS)*(t-muSS)**2*(t-q**6)
assert sp.factor(Wrs.charpoly(t).as_expr() - cRS) == 0
assert sp.factor(Wss.charpoly(t).as_expr() - cSS) == 0

# gap identity (Lemma 6) and dominance identities (Lemma 5)
g = 42*q**2-q**4-4*q**3-4*q-1
def zero(e): return sp.expand(e) == 0
assert zero(lamSS-lamRS - (1-q)**6*(q+1)**2*g/2048)
assert zero(lamRS-q**6
            - (1-q)**6*(q+1)**2*(q**4-8*q**3+30*q**2-8*q+1)/1024)
assert zero(lamRS-muRS - (1-q)**4*(5*q**2+2*q+5)
            *(3*q**6-30*q**5+93*q**4+124*q**3+93*q**2-30*q+3)/16384)

# sum-of-squares certificates from the proof of Lemma 5:
# positivity of the cofactors above, whence ||W_RS|| = lamRS
assert zero(q**4-8*q**3+30*q**2-8*q+1
            - ((q**2-4*q+1)**2 + 12*q**2))
assert zero(3*q**6-30*q**5+93*q**4+124*q**3+93*q**2-30*q+3
            - (3*q**4*((q-5)**2+6) + 124*q**3
               + 3*(31*q**2-10*q+1)))
print("all identities verified")
\end{verbatim}
}

\noindent The script confirms that $\Wrs$ and $\Wss$ are symmetric; that their characteristic
polynomials factor as $(t-\lambda_{\mathrm{RS}})(t-\mu_{\mathrm{RS}})^2(t-q^6)$ and
$(t-\lambda_{\mathrm{SS}})(t-\mu_{\mathrm{SS}})^2(t-q^6)$, giving the eigenvalue lists of
Lemma~\ref{lem:eig}; that the gap identity~\eqref{eq:gap} and the dominance
identities~\eqref{eq:rsdom1}--\eqref{eq:rsdom2} hold; and that the two sum-of-squares
certificates in the proof of Lemma~\ref{lem:eig} (which yield $\norm{\Wrs}=\lambda_{\mathrm{RS}}$)
are exact. Being exact symbolic identities in $q$, these are proofs of the corresponding statements.


\begin{thebibliography}{10}

\bibitem{YSJ21open}
C.~Yun, S.~Sra, and A.~Jadbabaie.
\newblock Open problem: Can single-shuffle SGD be better than reshuffling SGD and GD?
\newblock In \emph{Conference on Learning Theory (COLT)}, vol.\ 134, pp.\ 4653--4658, 2021.

\bibitem{RechtRe12}
B.~Recht and C.~R\'e.
\newblock Toward a noncommutative arithmetic-geometric mean inequality: conjectures, case-studies, and
consequences.
\newblock In \emph{Conference on Learning Theory (COLT)}, pp.\ 11.1--11.24, 2012.

\bibitem{LaiLim20}
Z.~Lai and L.-H.~Lim.
\newblock Recht--R\'e noncommutative arithmetic-geometric mean conjecture is false.
\newblock In \emph{International Conference on Machine Learning (ICML)}, 2020.

\bibitem{DeSa20}
C.~De~Sa.
\newblock Random reshuffling is not always better.
\newblock In \emph{Advances in Neural Information Processing Systems (NeurIPS)}, vol.\ 33, 2020.

\bibitem{Zhang18}
T.~Zhang.
\newblock A note on the matrix arithmetic-geometric mean inequality.
\newblock \emph{Electronic Journal of Linear Algebra}, 34:283--287, 2018.

\bibitem{HaochenSra19}
J.~Haochen and S.~Sra.
\newblock Random shuffling beats SGD after finite epochs.
\newblock In \emph{International Conference on Machine Learning (ICML)}, pp.\ 2624--2633, 2019.

\bibitem{Nagaraj19}
D.~Nagaraj, P.~Jain, and P.~Netrapalli.
\newblock SGD without replacement: Sharper rates for general smooth convex functions.
\newblock In \emph{International Conference on Machine Learning (ICML)}, pp.\ 4703--4711, 2019.

\bibitem{Rajput20}
S.~Rajput, A.~Gupta, and D.~Papailiopoulos.
\newblock Closing the convergence gap of SGD without replacement.
\newblock In \emph{International Conference on Machine Learning (ICML)}, 2020.

\bibitem{Ahn20}
K.~Ahn, C.~Yun, and S.~Sra.
\newblock SGD with shuffling: optimal rates without component convexity and large epoch requirements.
\newblock In \emph{Advances in Neural Information Processing Systems (NeurIPS)}, vol.\ 33, 2020.

\bibitem{Mishchenko20}
K.~Mishchenko, A.~Khaled, and P.~Richt\'arik.
\newblock Random reshuffling: Simple analysis with vast improvements.
\newblock \emph{arXiv preprint arXiv:2006.05988}, 2020.

\bibitem{SafranShamir20}
I.~Safran and O.~Shamir.
\newblock How good is SGD with random shuffling?
\newblock In \emph{Conference on Learning Theory (COLT)}, pp.\ 3250--3284, 2020.

\end{thebibliography}
\end{document}